\documentclass[11pt]{article}

\usepackage[a4paper,margin=1in]{geometry}
\usepackage{amsmath,amssymb,amsthm,mathtools}
\usepackage{enumitem}
\usepackage{needspace}
\usepackage{microtype}
\usepackage[T1]{fontenc}
\usepackage{lmodern}
\usepackage[hidelinks]{hyperref}
\usepackage[nameinlink,capitalize,noabbrev]{cleveref}
\allowdisplaybreaks[2]

\newtheorem{theorem}{Theorem}[section]
\newtheorem{lemma}[theorem]{Lemma}
\newtheorem{proposition}[theorem]{Proposition}
\newtheorem{corollary}[theorem]{Corollary}
\theoremstyle{definition}
\newtheorem{problem}[theorem]{Problem}
\theoremstyle{remark}
\newtheorem{remark}[theorem]{Remark}

\newcommand{\C}{\mathbb C}
\newcommand{\R}{\mathbb R}
\newcommand{\Q}{\mathbb Q}
\newcommand{\Z}{\mathbb Z}
\newcommand{\N}{\mathbb N}
\newcommand{\T}{\mathsf T}
\newcommand{\supp}{\operatorname{supp}}
\newcommand{\spec}{\operatorname{Spec}}
\newcommand{\ee}{\mathrm e}
\newcommand{\ii}{\mathrm i}
\newcommand{\1}{\mathbf 1}

\title{Perfect state transfer under matrix powers:
parity and spectral arithmetic}
\author{
  Xingkun Song$^{1,2}$\thanks{Corresponding author. Email:
  \href{mailto:xksong@126.com}{xksong@126.com}}\\[2ex]
  {\small $^{1}$ School of Mathematics and Statistics, Qinghai Minzu University,}\\
  {\small Xining, Qinghai 810007, P.R. China}\\[1ex]
  {\small $^{2}$ Qinghai Institute of Applied Mathematics,}\\
  {\small Xining, Qinghai 810007, P.R. China}
}
\date{}
\hypersetup{pdftitle={Perfect state transfer under matrix powers: parity and spectral arithmetic},pdfauthor={Xingkun Song}}

\begin{document}

\maketitle

\begin{abstract}
For a real symmetric matrix $H$ and distinct vertices $a,b$, we classify
exponents $k$ for which $H^k$ has perfect state transfer (PST) from $a$ to
$b$. If their supported eigenvalues are integer multiples of a common
positive number, every odd exponent reduces to $H$ and every positive even
exponent reduces to $H^2$. We determine the minimum transfer times using
a greatest common divisor of supported spectral differences.
For rational symmetric matrices, symmetry of the source vertex support
about zero implies the odd-power equivalence without a commensurability
assumption; this includes all bipartite graphs. If the source vertex supports
zero, PST under one positive even power implies PST under every positive
even power. For a symmetric three-point quadratic spectrum whose outer
projection signs agree and differ from the central sign, a nonzero rational
shift leaves exactly one PST exponent. We classify all adjacency powers of
hypercubes, cycles, and Johnson graphs, and all adjacency squares of paths.
In particular, the adjacency matrix of $P_7$ has PST from vertex $2$ to
vertex $6$ only at exponent $2$.
\end{abstract}



\section{Introduction}

Throughout, $\N=\{0,1,2,\ldots\}$ and $\N^+=\{1,2,3,\ldots\}$.
Let $G$ be a finite simple graph with adjacency matrix $A$.  For $k\in\N^+$,
consider the continuous-time walk generated by $A^k$,
\[
        U_k(t)=\exp(-\ii tA^k).
\]
Perfect state transfer (PST) from $a$ to $b$ at time $\tau>0$ means that
\[
        U_k(\tau)e_a=\gamma e_b
        \qquad\text{for some }\gamma\in\C,\quad |\gamma|=1.
\]
Here $e_v$ denotes the standard basis vector indexed by the vertex $v$.
We associate with the ordered pair $(a,b)$ the set
\[
 \mathcal K(A(G);a,b)=
 \{k\in\N^+:A(G)^k\text{ has PST from }a\text{ to }b\}.
\]
The entries of $A^k$ count walks of length $k$, so the matrices
$A,A^2,A^3,\ldots$
retain direct combinatorial information about $G$.  They are generally
weighted matrices with diagonal entries and long-range couplings; replacing
$A$ by $A^k$ is therefore different from allowing the original walk to evolve
for a longer time.

The adjacency model arose in quantum communication through spin
networks~\cite{Bose2003,ChristandlEtAl2004,ChristandlEtAl2005}.  Its graph-theoretic
development relates PST to spectral idempotents, strong cospectrality, and
arithmetic restrictions on the supported
eigenvalues~\cite{Godsil2011Periodic,Godsil2012,Godsil2012When,GodsilSmith2024}.
The problem considered here fixes the graph and varies the exponent of its
adjacency matrix.  It asks how much of the PST structure of $A$ survives under
these combinatorially defined matrix powers, and whether a power can create
PST that is absent for $A$ itself.

Polynomial Hamiltonians with non-nearest-neighbor couplings have already
been studied in distance-regular spin networks~\cite{JafarizadehSufiani2008}
and next-to-nearest-neighbor models~\cite{ChristandlVinetZhedanov2017}.
For the Johnson scheme, Ahmadi, Shirdareh Haghighi, and
Mokhtar~\cite{AhmadiEtAl2020} classified its individual unweighted relations,
and Vinet and Zhan~\cite{VinetZhan2020} characterized arbitrary real-weighted
members of the scheme. Thus polynomial Hamiltonians and their spectral phase
criteria are established tools. The contribution here is the uniform
classification of the restricted family $H,H^2,H^3,\ldots$, together with
conditions under which a successful exponent persists or is isolated.

The proofs are most transparent for a real symmetric matrix $H$ indexed by
the vertices of $G$, which will be specialized to $A(G)$ in the graph
applications.  Write
\[
        H=\sum_{\lambda\in\Phi}\lambda E_\lambda,
        \qquad \Phi=\spec(H),
\]
and, for a vector $x$, put
\[
        \supp_H(x)=\{\lambda\in\Phi:E_\lambda x\ne0\}.
\]
Vertices $a,b$ are strongly cospectral for $H$ if their eigenvalue supports
coincide and
\[
        E_\lambda e_b=\sigma_\lambda E_\lambda e_a,
        \qquad \sigma_\lambda\in\{1,-1\},
\]
on the common support.  Strong cospectrality is necessary for PST\@.  We also
write
\[
 \mathcal K(H;a,b)=
 \{k\in\N^+:H^k\text{ has PST from }a\text{ to }b\}.
\]

Godsil, Kirkland, and Monterde~\cite{GodsilKirklandMonterde2025} characterized
strong cospectrality of real pure states using a symmetric orthogonal
involution that is a polynomial in a fixed real symmetric matrix, and they gave
spectral conditions for PST between such states.  Our viewpoint is different:
the matrix and transfer pair are fixed, the monomial $x^k$ varies, and the
object to be classified is the exponent set $\mathcal K(H;a,b)$.  We first
determine how taking powers changes the spectral idempotents, then analyze the
phase identities $\ee^{-\ii\tau\mu}=\gamma\sigma_\mu$, and finally combine
these arguments to prove the parity theorem.

The spectral and arithmetic parts of this problem behave differently.  At the
spectral level, no arithmetic hypothesis is needed.  If $a,b$ are strongly
cospectral for $H$, then every odd power preserves their spectral idempotents.
Every positive even power instead merges the idempotents at $\rho$ and
$-\rho$.  Consequently, \cref{thm:strong-cospectral-exponents} proves that
\[
 \mathcal G(H;a,b)=
 \{k\in\N^+:a,b\text{ are strongly cospectral for }H^k\}
\]
is empty, $2\N+1$, or $\N^+$.  The arithmetic hypothesis enters only when
these phase identities are compared across the supported eigenvalues.

Our main result gives a parity dichotomy when all eigenvalues in
$\supp_H(e_a)\cup\supp_H(e_b)$ are integer multiples of a common positive
number.

\Needspace{15\baselineskip}
\begin{theorem}\label{thm:parity-dichotomy}
Let $H$ be a real symmetric matrix indexed by a finite set $V$, let
$a,b\in V$ be distinct, and suppose that
\[
        \supp_H(e_a)\cup\supp_H(e_b)\subseteq c\Z
        \qquad\text{for some }c>0.
\]
For every $k\in\N^+$,
\begin{equation}\label{eq:odd-even-reduction}
 H^k\text{ has PST from }a\text{ to }b
 \quad\Longleftrightarrow\quad
 \begin{cases}
 H\text{ has PST from }a\text{ to }b,&k\text{ odd},\\
 H^2\text{ has PST from }a\text{ to }b,&k\text{ even}.
 \end{cases}
\end{equation}
Consequently,
\[
 \mathcal K(H;a,b)\in
 \{\varnothing,\ 2\N+1,\ 2\N^+,\ \N^+\}.
\]
\end{theorem}

The same argument also determines the minimum positive transfer time.

\begin{corollary}\label{cor:minimum-transfer-time}
Under the hypotheses of \cref{thm:parity-dichotomy}, let
$k\in\mathcal K(H;a,b)$ and put $K=H/c$.  Choose
$\eta_0\in\supp_{K^k}(e_a)$ and define
\[
 \delta_k=\gcd\{\lvert\eta-\eta_0\rvert:
       \eta\in\supp_{K^k}(e_a)\setminus\{\eta_0\}\}.
\]
The integer $\delta_k$ is independent of the choice of $\eta_0$.  The minimum
positive transfer time under $H^k$ is
\[
        \frac{\pi}{c^k\delta_k}.
\]
\end{corollary}

\begin{remark}\label{rem:normalization-independent}
The number $c$ in \cref{thm:parity-dichotomy} need not be unique.  For any
admissible choice of $c$, the product $c^k\delta_k$ is the least positive element
of the additive subgroup of $\R$ generated by the differences between the
eigenvalues in $\supp_{H^k}(e_a)$.  Hence $c^k\delta_k$, and therefore the
minimum positive transfer time in \cref{cor:minimum-transfer-time}, is
independent of the normalization.
\end{remark}

Only the eigenvalues in $\supp_H(e_a)\cup\supp_H(e_b)$ are required to lie in
$c\Z$; no assumption is made on the other eigenvalues of $H$.

The proof separates spectral collisions from arithmetic. Odd powers have
singleton fibers on the real line; even powers identify only opposite
eigenvalues. After removing the common power of $2$ from the integer support,
we choose an odd reference eigenvalue. A difference from an even eigenvalue
has valuation zero for every exponent. Between odd eigenvalues, odd powers
preserve the valuation, whereas every even power adds the same quantity
$\nu_2(k)-1$ to the valuation of the square difference. This proves the
parity reduction without an iteration over successive squarings.

For rational matrices, spectral arithmetic can sometimes be deduced from
PST rather than assumed.

\begin{theorem}\label{thm:bipartite-odd}
Let $H$ be a rational symmetric matrix and suppose
$\supp_H(e_a)=-\supp_H(e_a)$. For every vertex $b\ne a$ and
every positive odd integer $r$,
\[
 H^r\text{ has PST from }a\text{ to }b
 \quad\Longleftrightarrow\quad
 H\text{ has PST from }a\text{ to }b.
\]
In particular, this holds for the adjacency matrix of every finite bipartite
graph, without a commensurability assumption on its eigenvalue support.
\end{theorem}

\begin{theorem}\label{thm:zero-even}
Let $H$ be a rational symmetric matrix and suppose
$0\in\supp_H(e_a)$. For any distinct vertex $b$, either no positive even
power of $H$ has PST from $a$ to $b$, or every positive even power does.
If one does, then $\lambda^2\in\Q$ for every
$\lambda\in\supp_H(e_a)$.
\end{theorem}

Theorems~\ref{thm:bipartite-odd} and~\ref{thm:zero-even} are proved in
\cref{sec:rational}. They identify two situations in which isolated
exponents cannot occur.

The assumption on $\supp_H(e_a)\cup\supp_H(e_b)$ cannot be omitted.  For the
path $P_7$, \cref{thm:P7-exponents} proves
\[
        \mathcal K(A(P_7);2,6)=\{2\}.
\]
Thus different even powers need not have the same PST behavior when the
inclusion in \cref{thm:parity-dichotomy} fails.

The graph applications determine the relevant powers for several standard
families.  Nonbinary Hamming graphs admit no PST under an adjacency power.
For hypercubes, \cref{thm:hypercube} determines the answer from the parity of
the dimension.  For cycles, \cref{thm:cycle-powers} leaves $C_4$ under every
positive power and $C_8$ under the even powers.  For paths,
\cref{thm:path-classification} classifies PST under $A^2$.  For Johnson graphs,
\cref{thm:johnson-powers} leaves only $J(2,1)$ under odd powers and $J(6,3)$
under even powers.

The spectral idempotents of matrix powers are described in \cref{sec:powers};
\cref{sec:arithmetic} proves \cref{thm:parity-dichotomy}, and
\cref{sec:rational} develops the rational-matrix refinements.
The graph classifications are collected in \cref{sec:applications}, while
\cref{sec:boundary} isolates the path with
$\mathcal K(A(P_7);2,6)=\{2\}$.
The appendix contains the association-scheme arguments used for Hamming and
Johnson graphs and the proof of \cref{thm:johnson-powers}.

\section{Spectral idempotents of matrix powers}\label{sec:powers}

Let $H$ be a real symmetric matrix with spectral decomposition
\[
        H=\sum_{\lambda\in\Phi}\lambda E_\lambda.
\]
For a real-valued function $f$ on $\Phi$, put
\[
 f(H)=\sum_{\lambda\in\Phi}f(\lambda)E_\lambda,
 \qquad
 F_\mu=\sum_{\lambda:\,f(\lambda)=\mu}E_\lambda.
\]
For $\mu\in f(\Phi)$, the fiber (inverse-image set) of $f$ over $\mu$ is
\[
        f^{-1}(\{\mu\})=\{\lambda\in\Phi:f(\lambda)=\mu\}.
\]
The $F_\mu$ are the spectral idempotents of $f(H)$. We record the
standard functional-calculus criterion, including possible spectral collisions.

\begin{lemma}\label{lem:scalar-transformations}
If $H=cK$ with $c>0$, then
\[
 \exp(-\ii tH^k)=\exp\bigl(-\ii(tc^k)K^k\bigr).
\]
Also, for $\alpha\in\R$,
\[
 \exp\bigl(-\ii t(H+\alpha I)\bigr)
 =\ee^{-\ii t\alpha}\exp(-\ii tH).
\]
Thus scalar rescaling changes transfer times by the factor $c^{-k}$, while a
scalar shift changes only the global phase.
\end{lemma}

\begin{proof}
The first identity follows from $H^k=c^kK^k$; the second follows because
$H$ commutes with $\alpha I$.
\end{proof}

\begin{proposition}\label{prop:spectral-function-pst}
The matrix $f(H)$ has PST from $a$ to $b$ if and only if the following three
conditions hold.
\begin{enumerate}[label=\textup{(\roman*)}]
\item The vertices $a,b$ are strongly cospectral for $H$.
\item If $\lambda$ and $\theta$ belong to their common eigenvalue support and
      $f(\lambda)=f(\theta)$, then $\sigma_\lambda=\sigma_\theta$.
\item There are $\tau>0$ and $\gamma\in\C$, with $|\gamma|=1$, such that
      \[
             \ee^{-\ii\tau\mu}=\gamma\sigma_\mu
      \]
      whenever $F_\mu e_a\ne0$, where $\sigma_\mu$ is the common sign from
      \textup{(ii)}.
\end{enumerate}
\end{proposition}

\begin{proof}
If $\exp(-\ii\tau f(H))e_a=\gamma e_b$, applying $F_\mu$ gives
\[
 \ee^{-\ii\tau\mu}F_\mu e_a=\gamma F_\mu e_b.
\]
The two vectors are real, so on every supported fiber their ratio is a sign
$\sigma_\mu\in\{1,-1\}$. Applying $E_\lambda$ for
$f(\lambda)=\mu$ gives
$E_\lambda e_b=\sigma_\mu E_\lambda e_a$.
An unsupported fiber vanishes at both vertices, and orthogonality implies
that all its constituent projections vanish at both vertices. This proves
\textup{(i)}--\textup{(iii)}. Conversely, summing the projected identities
over the supported fibers proves the PST equation.
\end{proof}

For even powers, let
\[
        \mathcal R(H)=\{|\lambda|:\lambda\in\spec(H)\},
        \qquad
        F_\rho=E_\rho+E_{-\rho}\quad(\rho>0),
        \qquad F_0=E_0,
\]
where a missing idempotent is interpreted as zero.

\begin{proposition}\label{prop:power-fibers}
Let $a,b$ be distinct vertices.
\begin{enumerate}[label=\textup{(\roman*)}]
\item If $k$ is odd, then $H^k$ has PST from $a$ to $b$ if and only if
$a,b$ are strongly cospectral for $H$ and
\[
 \ee^{-\ii\tau\lambda^k}=\gamma\sigma_\lambda
\]
on their common support for some $\tau>0$ and $|\gamma|=1$.
\item If $k$ is positive and even, then $H^k$ has PST from $a$ to $b$ if and
only if $a,b$ are strongly cospectral for $H^2$ and, writing
\[
 F_\rho e_b=\sigma_\rho F_\rho e_a
\]
for every $\rho\in\mathcal R(H)$ with $F_\rho e_a\ne0$, there are
$\tau>0$ and $|\gamma|=1$ such that
\[
 \ee^{-\ii\tau\rho^k}=\gamma\sigma_\rho
\]
for every such $\rho$.  Equivalently, the values $\rho^2$ form their common
support for $H^2$.
\end{enumerate}
\end{proposition}

\begin{proof}
For odd $k$, every fiber of $x\mapsto x^k$ on $\spec(H)$ is a singleton, so
condition~\textup{(ii)} of \cref{prop:spectral-function-pst} is vacuous.

For positive even $k$, the fibers are
$\spec(H)\cap\{\rho,-\rho\}$ for $\rho>0$ and $\{0\}$.  Their idempotents are
\[
 F_\rho=E_\rho+E_{-\rho}\quad(\rho>0),\qquad F_0=E_0,
\]
which are exactly the spectral idempotents of $H^2$.
Conditions~\textup{(i)} and~\textup{(ii)} of
\cref{prop:spectral-function-pst} are therefore equivalent to strong
cospectrality for $H^2$ and exclude extra support of $e_b$ on another fiber.
Condition~\textup{(iii)} is precisely the displayed phase identity.
\end{proof}

All positive even powers therefore have the same spectral idempotents as
$H^2$.  The strong-cospectrality condition for the spectral idempotents of
$H^2$ is independent of the even exponent, although the phase differences may
depend on the exponent.

\subsection{Preservation of strong cospectrality}

For a real symmetric matrix $H$ and distinct vertices $a,b$, write
\[
 \mathcal G(H;a,b)=
 \{k\in\N^+:\ a,b\text{ are strongly cospectral for }H^k\}.
\]

\begin{theorem}
\label{thm:strong-cospectral-exponents}
Let $H$ be a real symmetric matrix and let $a,b$ be distinct vertices.
If $a,b$ are not strongly cospectral for $H$, then
$\mathcal G(H;a,b)=\varnothing$.  Otherwise write
$E_\lambda e_b=\sigma_\lambda E_\lambda e_a$ on the common support.  The
compatibility condition for opposite eigenvalues is
\begin{equation}\label{eq:opposite-sign-compatibility}
\sigma_\rho=\sigma_{-\rho}
\quad\text{whenever both }\rho\text{ and }-\rho
\text{ occur in the common eigenvalue support of }a,b.
\end{equation}
Then
\[
\mathcal G(H;a,b)=
\begin{cases}
\N^+,&\text{if \eqref{eq:opposite-sign-compatibility} holds},\\
2\N+1,&\text{otherwise}.
\end{cases}
\]
\end{theorem}

\begin{proof}
Suppose that $a,b$ are strongly cospectral for $H^k$, and let $F_\mu$ denote
the spectral idempotent of $H^k$ for the eigenvalue $\mu$.  For every $\mu$ in
the common support, write
\begin{equation}\label{eq:power-idempotent-sign}
        F_\mu e_b=\varepsilon_\mu F_\mu e_a,
        \qquad \varepsilon_\mu\in\{1,-1\}.
\end{equation}
If $\mu$ is not in the common support, then
$F_\mu e_a=F_\mu e_b=0$.  Since the projections $E_\lambda$ with
$\lambda^k=\mu$ are pairwise orthogonal, this implies
$E_\lambda e_a=E_\lambda e_b=0$ for every such $\lambda$.

Now let $\mu$ be supported.  For every eigenvalue $\lambda$ of $H$ satisfying
$\lambda^k=\mu$, one has $E_\lambda F_\mu=E_\lambda$.  Applying $E_\lambda$
to \eqref{eq:power-idempotent-sign} gives
\[
        E_\lambda e_b=\varepsilon_\mu E_\lambda e_a.
\]
The supported and unsupported cases together show that $a,b$ are strongly
cospectral for $H$.

Now assume strong cospectrality for $H$.  Every odd power is injective on the
real spectrum and therefore preserves the spectral idempotents and their
signs.  Hence all odd positive integers belong to $\mathcal G(H;a,b)$.

For an even power, the nonzero spectral idempotent at absolute value $\rho$ is
$F_\rho=E_\rho+E_{-\rho}$.  On the merged eigenspace
$\operatorname{im}F_\rho$,
\[
 F_\rho e_b
 =\sigma_\rho E_\rho e_a+\sigma_{-\rho}E_{-\rho}e_a.
\]
If both eigenvalues occur in the common support, put
$x=E_\rho e_a$ and $y=E_{-\rho}e_a$.  These vectors are nonzero and
orthogonal.  An identity
$\sigma_\rho x+\sigma_{-\rho}y=\varepsilon(x+y)$, projected onto
$\operatorname{im}E_\rho$ and $\operatorname{im}E_{-\rho}$, gives
$\sigma_\rho=\varepsilon=\sigma_{-\rho}$.  The converse is immediate.
If only one eigenvalue is supported, there is no constraint; $E_0$ is
unchanged and imposes no additional compatibility condition.
All even powers have the spectral idempotents $F_\rho$.  Therefore either every
even exponent belongs to $\mathcal G(H;a,b)$ or no even exponent does.  The
formula for $\mathcal G(H;a,b)$ follows.
\end{proof}

For a real symmetric matrix $H$ and distinct vertices $a,b$,
\[
 \mathcal K(H;a,b)
 \subseteq
 \mathcal G(H;a,b).
\]
The inclusion separates the spectral requirement of strong cospectrality from
the phase identities $\ee^{-\ii\tau\mu}=\gamma\sigma_\mu$ required for PST\@.
Thus \cref{thm:strong-cospectral-exponents} isolates the purely spectral
obstruction; the remaining issue is the arithmetic compatibility of the
resulting pairwise congruences.

\section{Arithmetic phase conditions and parity}\label{sec:arithmetic}

On the common support, suppose that the projected vectors satisfy
\[
        E_\lambda e_b=\sigma_\lambda E_\lambda e_a,
        \qquad \sigma_\lambda\in\{1,-1\}.
\]
For PST at time $\tau$, these signs must satisfy
\begin{equation}\label{eq:supported-phase-identities}
        \ee^{-\ii\tau\lambda}=\gamma\sigma_\lambda
        \qquad\bigl(\lambda\in\supp_H(e_a)\bigr)
\end{equation}
for some $\gamma\in\C$ with $|\gamma|=1$.  Taking ratios in
\eqref{eq:supported-phase-identities} eliminates the global phase $\gamma$.
The resulting sign congruences apply to a real symmetric matrix.  The gcd and
$2$-adic arguments require the supported eigenvalues to be integers.
Godsil~\cite{Godsil2011Periodic,Godsil2012When} used the equivalent
phase-difference relations in his work on periodicity and state transfer.

We call a finite set of real numbers \emph{commensurate} if it is contained
in $c\Z$ for some $c>0$. For a finite set $S\subseteq\R$ containing a nonzero element,
\[
 S\subseteq c\Z\text{ for some }c>0
 \quad\Longleftrightarrow\quad
 \dim_{\Q}\operatorname{span}_{\Q}S=1.
\]
Indeed, the forward implication is immediate, and the reverse implication
follows by expressing all elements of $S$ as rational multiples of one
nonzero element and clearing denominators.

For a nonzero integer $n$, let $\nu_2(n)$ be the largest $u\in\N$
such that $2^u$ divides $n$, and set $\nu_2(0)=+\infty$.

\subsection{Phase differences}

\begin{theorem}\label{thm:phase-congruences}
Let $K$ be a real symmetric matrix, let $a,b$ be vertices, and suppose that, for each
$\mu\in\supp_K(e_a)$, one has
$E_\mu e_b=\sigma_\mu E_\mu e_a$ with $\sigma_\mu\in\{1,-1\}$.
The matrix $K$ has PST from $a$ to $b$ at time $\tau$ if
and only if, for all
supported $\mu,\nu$,
\begin{align}
 \tau(\mu-\nu)&\in2\pi\Z
 &&\text{if }\sigma_\mu=\sigma_\nu,
 \label{eq:phase-congruence-same}\\
 \tau(\mu-\nu)&\in\pi+2\pi\Z
 &&\text{if }\sigma_\mu=-\sigma_\nu.
 \label{eq:phase-congruence-opposite}
\end{align}
\end{theorem}

\begin{proof}
The assumed relations give
\[
 \sum_{\mu\in\supp_K(e_a)}\|E_\mu e_b\|^2
 =\sum_{\mu\in\supp_K(e_a)}\|E_\mu e_a\|^2=1.
\]
These are pairwise orthogonal projections of $e_b$, and their squared norms
already sum to $\|e_b\|^2=1$.  Hence all remaining projections vanish:
$E_\theta e_b=0$ for every
$\theta\notin\supp_K(e_a)$.  Thus $a$ and $b$
have the same eigenvalue support and are strongly cospectral for $K$.
By \cref{prop:spectral-function-pst}, applied to the identity function on the
spectrum of $K$, PST is equivalent to the concrete phase system
\begin{equation}\label{eq:phase-system-on-K}
        \ee^{-\ii\tau\mu}=\gamma\sigma_\mu
        \qquad\bigl(\mu\in\supp_K(e_a)\bigr).
\end{equation}
Dividing the identities in \eqref{eq:phase-system-on-K} for $\mu$ and $\nu$
gives
\begin{equation}\label{eq:phase-ratio-on-K}
        \ee^{-\ii\tau(\mu-\nu)}=\sigma_\mu\sigma_\nu,
\end{equation}
which is equivalent to \eqref{eq:phase-congruence-same} and
\eqref{eq:phase-congruence-opposite}.  Conversely, choose one supported $\nu$
and set
$\gamma=\ee^{-\ii\tau\nu}\sigma_\nu^{-1}
=\ee^{-\ii\tau\nu}\sigma_\nu$; the two congruences recover the full system
\eqref{eq:phase-system-on-K}.
\end{proof}

\subsection{Integer eigenvalues and parity}

When the supported eigenvalues are integers, their phase congruences are
controlled by a greatest common divisor and a parity condition.

\begin{theorem}\label{thm:integral-parity}
Let $K$ be a real symmetric matrix and let $a,b$ be distinct strongly cospectral
vertices.  Suppose their common eigenvalue support consists of distinct
integers $q_0,q_1,\ldots,q_m$, where $m\in\N^+$, and write
$E_{q_j}e_b=\sigma_jE_{q_j}e_a$.  Set
\[
        d=\gcd\{|q_j-q_0|:1\leq j\leq m\}\in\N^+.
\]
Then $K$ has PST from $a$ to $b$ if and only if
\begin{equation}\label{eq:gcd-parity}
 \frac{q_j-q_0}{d}\equiv
 \begin{cases}
 0\pmod2,&\sigma_j=\sigma_0,\\
 1\pmod2,&\sigma_j=-\sigma_0.
 \end{cases}
\end{equation}
When \eqref{eq:gcd-parity} holds, all positive transfer times are
\[
        \tau=\frac{(2\ell+1)\pi}{d},
        \qquad \ell\in\N.
\]
Equivalently, if
\[
        h=\min_{1\leq j\leq m}\nu_2(|q_j-q_0|),
\]
then PST occurs precisely when
\begin{align*}
 \nu_2(|q_j-q_0|)&>h &&\text{if }\sigma_j=\sigma_0,\\
 \nu_2(|q_j-q_0|)&=h &&\text{if }\sigma_j=-\sigma_0.
\end{align*}
\end{theorem}

\begin{proof}
By \cref{thm:phase-congruences}, the phase congruences relative to $q_0$ are
\begin{equation}\label{eq:integral-relative-phase-congruences}
 \tau(q_j-q_0)\in
 \begin{cases}
 2\pi\Z,&\sigma_j=\sigma_0,\\
 \pi+2\pi\Z,&\sigma_j=-\sigma_0.
 \end{cases}
\end{equation}
Put
\[
        n_j=\frac{q_j-q_0}{d},
        \qquad x=\frac{\tau d}{\pi}.
\]
Then $\gcd(n_1,\ldots,n_m)=1$, and
\eqref{eq:integral-relative-phase-congruences} says that every $xn_j$ is an
integer.  B\'ezout's identity therefore gives $x\in\Z$.  Because
$a\ne b$, both signs occur.  Indeed,
\[
 e_b=\sum_{j=0}^m\sigma_jE_{q_j}e_a;
\]
all signs $+1$ would give $e_b=e_a$, and all signs $-1$ would give
$e_b=-e_a$, both impossible for distinct standard basis vectors.
Hence $xn_j$ is odd for at least one $j$, and therefore $x$ is odd.  For odd
$x$, \eqref{eq:integral-relative-phase-congruences} is equivalent to
\[
 n_j\equiv0\pmod2\quad\text{when }\sigma_j=\sigma_0,
 \qquad
 n_j\equiv1\pmod2\quad\text{when }\sigma_j=-\sigma_0,
\]
which is \eqref{eq:gcd-parity}.  Conversely, if \eqref{eq:gcd-parity} holds,
then every odd integer $x=2\ell+1$ satisfies
\eqref{eq:integral-relative-phase-congruences}.  Therefore the positive
transfer times are exactly $(2\ell+1)\pi/d$.

Finally,
\[
        \nu_2(d)=\min_{1\leq j\leq m}\nu_2(|q_j-q_0|)=h.
\]
The quotient $(q_j-q_0)/d$ is odd exactly when
$\nu_2(|q_j-q_0|)=h$, which proves the valuation formulation in
\cref{thm:integral-parity}.
\end{proof}

\subsection{A unified valuation formula}

\begin{lemma}\label{lem:valuation-minima}
Let $u$ be an odd integer and let $v\ne u$ be an integer. For a positive
integer $k$, with $v\ne -u$ also required when $k$ is even,
\[
 \nu_2(v^k-u^k)=
 \begin{cases}
 0,&v\text{ even},\\
 \nu_2(v-u),&v\text{ odd},\ k\text{ odd},\\
 \nu_2(v-u)+\nu_2(v+u)+\nu_2(k)-1,
     &v\text{ odd},\ k\text{ even}.
 \end{cases}
\]
\end{lemma}
\begin{proof}
If $v$ is even, the difference is odd. For odd $v$ and odd $k$, factor
$v^k-u^k$ by $v-u$; the quotient is a sum of an odd number of odd terms.
For even $k=2^s r$ with $r$ odd, the same observation removes the odd
factor $r$. The remaining factorization is
\[
 v^{2^s}-u^{2^s}
 =(v-u)(v+u)\prod_{j=1}^{s-1}(v^{2^j}+u^{2^j}).
\]
Every factor in the product is $2$ modulo $4$, giving the formula.
\end{proof}

\begin{proof}[Proof of \cref{thm:parity-dichotomy}]
If $a,b$ are not strongly cospectral for $H$, no power has PST by
\cref{thm:strong-cospectral-exponents}. Otherwise their common support
contains a nonzero eigenvalue, since a singleton support would force
$e_b=\pm e_a$. Scaling does not affect PST. We may therefore assume that
the supported eigenvalues are integers and, after dividing by their common
power of $2$, that at least one is odd. Choose an odd supported eigenvalue
$u$ as reference.

Odd powers preserve all spectral idempotents and their signs. By
\cref{lem:valuation-minima}, the valuations of all nonzero supported
differences from $u$ are unchanged by every odd power. The valuation
criterion in \cref{thm:integral-parity} therefore proves the odd equivalence.

For even powers, if the signs at opposite supported eigenvalues disagree,
none has PST. Otherwise merge these pairs and take one nonnegative
representative of each absolute value, choosing an odd one as reference.
If there is only one representative, no even power can have PST between
distinct vertices. Otherwise all even powers have these same merged signs.
If an even representative
occurs, precisely the differences to even representatives have minimum
valuation zero, for every even exponent. If every representative is odd,
\cref{lem:valuation-minima} gives
\[
 \nu_2(v^k-u^k)=\nu_2(v^2-u^2)+\nu_2(k)-1.
\]
Thus all differences receive the same additive shift, and the indices
attaining the minimum do not change. The integral parity criterion proves
that every even exponent gives the same answer as $2$.
\end{proof}

\begin{proof}[Proof of \cref{cor:minimum-transfer-time}]
Since $a\ne b$ and $K^k$ has PST from $a$ to $b$, the eigenvalue support
$\supp_{K^k}(e_a)$ contains at least two distinct eigenvalues.  Indeed, a
one-element support together with strong cospectrality would give
$e_b=\pm e_a$.
For a finite set $S$ of integers,
\[
 \gcd\{\eta-\eta_0:\eta\in S\}
 =\gcd\{\eta-\xi:\eta,\xi\in S\}.
\]
Every pairwise difference is the difference of two fixed-point differences,
and each fixed-point difference is itself a pairwise difference.  Thus the
gcd is independent of the fixed element.  Applying this observation to
$\supp_{K^k}(e_a)$ proves that $\delta_k$ is independent of $\eta_0$.

By \cref{thm:integral-parity}, the minimum positive transfer time for $K^k$ is
$t_K=\pi/\delta_k$.  The identity
$\exp(-\ii tH^k)=\exp(-\ii(tc^k)K^k)$ shows that the corresponding time for
$H^k$ is
\[
        t_H=\frac{t_K}{c^k}=\frac{\pi}{c^k\delta_k}.
\]
\end{proof}

All four possibilities in \cref{thm:parity-dichotomy} occur: distinct
vertices of $K_3$ give the empty set, $K_2$ gives the positive odd integers,
$J(6,3)$ gives the positive even integers between complementary vertices,
and antipodes of $C_4$ give all positive integers.  These graph examples are
proved in \cref{sec:applications}.  If
$\supp_H(e_a)\cup\supp_H(e_b)$ is not contained in $c\Z$ for any $c>0$,
further exponent sets can occur;
\cref{thm:P7-exponents} gives
$\mathcal K(A(P_7);2,6)=\{2\}$.

\section{Rational matrices and spectral descent}\label{sec:rational}

The next argument uses the same Galois-invariance principle underlying
the arithmetic theory of periodic vertices~\cite{Godsil2011Periodic,Godsil2012When}.
Here it is applied before taking a root of a supported eigenvalue.

\begin{lemma}\label{lem:galois-descent}
Let $H$ be a rational symmetric matrix and let $S=\supp_H(e_a)$.
The set $S$ is invariant under algebraic conjugation. Suppose $S$ contains
a nonzero element and, for some positive integer $r$, the numbers
$\lambda^{2r}$, $\lambda\in S\setminus\{0\}$, have rational pairwise
ratios. Then $\lambda^2\in\Q$ for every $\lambda\in S$.
\end{lemma}
\begin{proof}
The resolvent entry
\[
 e_a^{\T}(zI-H)^{-1}e_a
 =\sum_{\lambda\in S}\frac{\|E_\lambda e_a\|^2}{z-\lambda}
\]
is a rational function over $\Q$ whose poles are exactly $S$; the residues
are positive and cannot cancel. Its reduced denominator lies in $\Q[z]$ and has the distinct roots in
$S$, so algebraic conjugation permutes $S$. Consequently
$T=\sum_{\lambda\in S}\lambda^{2r}$
is rational and positive. Fix a nonzero $\mu\in S$. Since
$T/\mu^{2r}$ is a positive rational number, $\mu^{2r}\in\Q$.
Every conjugate of $\mu^2$ is nonnegative, because all conjugates of $\mu$
are eigenvalues of the real symmetric matrix $H$. All these conjugates
have the same $r$th power. Injectivity on $[0,\infty)$ and separability
imply that $\mu^2\in\Q$.
\end{proof}

\begin{proof}[Proof of \cref{thm:bipartite-odd}]
Put $S=\supp_H(e_a)$, which is symmetric about zero by hypothesis.
Suppose an odd power $H^r$ has PST. Then $e_a$ is periodic for $H^r$.
Indeed, symmetry of the unitary matrix gives transfer back with the same
phase, so $\exp(-2\ii\tau H^r)e_a=\gamma^2e_a$ at twice a transfer time.
Periodicity and the supported pair $\lambda,-\lambda$ imply that
$t\lambda^r\in\pi\Z$ for a positive period $t$. Hence
$\lambda^r/\mu^r\in\Q$ for all nonzero $\lambda,\mu\in S$.
By \cref{lem:galois-descent}, $\lambda^2,\mu^2\in\Q$, and oddness gives
\[
 \frac{\lambda}{\mu}
 =\frac{\lambda^r/\mu^r}{(\lambda^2/\mu^2)^{(r-1)/2}}\in\Q.
\]
Thus $S\subseteq c\Z$ for some $c>0$. Strong cospectrality gives the
same support at $b$, and \cref{thm:parity-dichotomy} yields PST for $H$.
Conversely, the same argument with $r=1$ shows that PST for $H$ forces
$S\subseteq c\Z$; the parity theorem then gives PST for every odd power.
Finally, if a diagonal sign matrix $D$ satisfies $DHD=-H$, its action
identifies the projections at $\lambda$ and $-\lambda$ and fixes $e_a$
up to sign. This proves the support symmetry for every rational weighted
bipartite matrix and hence for every bipartite graph.
\end{proof}

\begin{proof}[Proof of \cref{thm:zero-even}]
Suppose $H^{2r}$ has PST for some $r\geq1$. Its periodicity and the
supported eigenvalue zero imply that every nonzero $\lambda^{2r}$ is an
integer multiple of $2\pi/t$, for a positive period $t$. The preceding
lemma gives $\lambda^2\in\Q$ throughout the support.
After a positive rational scaling, the supported eigenvalues of $H^2$
are distinct nonnegative integers $q_0=0,q_1,\ldots,q_m$.
All its positive powers have the same spectral idempotents and signs,
and
\[
 \nu_2(q_j^s-0)=s\nu_2(q_j)\qquad(s\geq1).
\]
The set of minimum-valuation indices is independent of $s$.
By \cref{thm:integral-parity}, PST for one positive power of $H^2$
is equivalent to PST for all of them.
\end{proof}

\begin{corollary}\label{cor:trees-odd}
For a finite tree $T$, an odd power of $A(T)$ has PST between distinct
vertices if and only if $T=P_2$ or $P_3$ and the vertices are its endpoints.
In either case every positive odd exponent works.
\end{corollary}
\begin{proof}
Trees are bipartite. Apply \cref{thm:bipartite-odd} and the adjacency-PST
classification of trees by Coutinho, Juliano, and
Spier~\cite{CoutinhoJulianoSpier2024}.
\end{proof}

\begin{remark}\label{rem:irrational-weights}
Rationality cannot be dropped from \cref{thm:bipartite-odd}. Let
$B=A(Q_3)$ and define $H$ by the real spectral cube root of $B$.
The odd spectral function preserves $DHD=-H$, and $H^3=B$ has antipodal
PST. However, the support for $H$ at every vertex contains
$\pm1$ and $\pm\sqrt[3]{3}$. Periodicity would force their positive
ratio to be rational, which is impossible. Thus $H$ has no PST.
\end{remark}

\section{Graph applications}\label{sec:applications}

We now specialize to ordinary adjacency matrices.  The results in
\cref{sec:applications} determine the PST exponent sets for hypercubes and cycles, classify
PST under $A^2$ on paths, and determine all adjacency powers
with PST in Johnson graphs.  General association-scheme arguments used in
some of the proofs are collected in \cref{app:association-schemes}.

The first consequence for graph adjacency matrices is the following.

\begin{corollary}\label{cor:graph-exponent-sets}
Let $G$ be a finite simple graph, let $A=A(G)$, and let $a,b$ be distinct vertices.
If
\[
 \supp_A(e_a)\cup\supp_A(e_b)\subseteq c\Z
 \qquad\text{for some }c>0,
\]
then
\[
 \mathcal K(A(G);a,b)\in
 \{\varnothing,\ 2\N+1,\ 2\N^+,\ \N^+\}.
\]
In particular, all odd adjacency exponents reduce to $A$, and all positive
even adjacency exponents reduce to $A^2$.
\end{corollary}

\begin{proof}
Apply \cref{thm:parity-dichotomy} to $H=A(G)$.
\end{proof}

The hypothesis in \cref{cor:graph-exponent-sets} concerns only the eigenvalue
supports of $e_a$ and $e_b$.  The support hypothesis holds with $c=1$ for an
integral graph.  More generally, the support hypothesis holds whenever
multiplying every eigenvalue
in $\supp_A(e_a)\cup\supp_A(e_b)$ by one positive number gives an integer.

\subsection{Hamming graphs and hypercubes}

For $d,q\in\N^+$ with $q\geq2$, the vertices of $H(d,q)$ are the $q$-ary
words of length $d$.  Two words are adjacent when their Hamming distance is
one.  Its distance-$\ell$ relation has valency
\[
        k_\ell=\binom d\ell(q-1)^\ell.
\]

\begin{proposition}\label{prop:hamming-adjacency-powers}
Let $A=A(H(d,q))$.  If $q\geq3$, then no positive power of $A$ has PST
between distinct vertices.  If $q=2$, then PST under a positive power of
$A$ can occur between distinct vertices only when they are antipodal.
\end{proposition}

\begin{proof}
Apply \cref{thm:nonbinary-hamming}, proved in
\cref{app:hamming-schemes}, to $K=A$.  It excludes every distinct pair when
$q\ge3$; when $q=2$, PST implies strong cospectrality and hence forces the
pair to be antipodal.
\end{proof}

Bernasconi, Godsil, and Severini~\cite{BernasconiGodsilSeverini2008}, Cheung and
Godsil~\cite{CheungGodsil2011}, and Coutinho and
Godsil~\cite{CoutinhoGodsil2016} studied PST on cubelike graphs, including the
adjacency walk on the hypercube.
\Cref{thm:hypercube} gives the classification for every adjacency power of the
binary hypercube.  Write $\mathbb F_2=\mathbb Z/2\mathbb Z$.  Let $Q_d$
be the graph on $\mathbb{F}_2^d$ in which two words are
adjacent when they differ in one coordinate, and let $\bar x=x+\1$ denote the
antipode of $x$, where $\1=(1,\ldots,1)$ in $\mathbb{F}_2^d$.

\begin{theorem}
\label{thm:hypercube}
Let $d,k\in\N^+$ and let $A=A(Q_d)$.  Then $A^k$ has PST between
every antipodal pair $x,\bar x$ if and only if
\[
        d\text{ is even}\qquad\text{or}\qquad k\text{ is odd}.
\]
More precisely,
\begin{enumerate}[label=\textup{(\roman*)}]
\item if $d$ is even, every $k\in\N^+$ has PST, with minimum positive transfer time
      $\pi/2^k$;
\item if $d$ is odd, exactly the odd exponents have PST, with minimum
      positive transfer time $\pi/2$.
\end{enumerate}
\end{theorem}

\begin{proof}
The eigenspaces of $Q_d$ are indexed by the character weights
$j=0,\ldots,d$.  The corresponding adjacency eigenvalue is
\[
        \theta_j=d-2j,
\]
and we denote its spectral idempotent by $E_j$.  Fourier character evaluation
at the antipode gives
\begin{equation}\label{eq:hypercube-sign}
        E_j e_{\bar x}=(-1)^jE_j e_x.
\end{equation}
The $E_j$ and $\theta_j$ are the standard primitive idempotents and eigenvalues of the binary
Hamming scheme; see, for example,~\cite{CoutinhoEtAl2015}.  Since the
primitive idempotents have positive constant diagonal, $E_je_x\ne0$ for every
$j$; thus all displayed signs occur in the support.

At time $\pi/2$, the phases for the ordinary adjacency matrix satisfy
\[
 \ee^{-\ii\pi\theta_j/2}
 =\ee^{-\ii\pi d/2}(-1)^j.
\]
Thus $A$ has PST between every antipodal pair for every $d$.

It remains to determine PST under $A^2$.  If $d$ is odd, then
$\theta_{d-j}=-\theta_j$, while the two signs in
\eqref{eq:hypercube-sign} are opposite.  Hence the square merges two
idempotents with incompatible signs, and \cref{prop:power-fibers} excludes
PST\@.
If $d$ is even, then write $\theta_j=2(d/2-j)$.  At time $\pi/4$,
\[
 \ee^{-\ii\pi\theta_j^2/4}
 =(-1)^{(d/2-j)^2}
 =(-1)^{d/2}(-1)^j,
\]
so $A^2$ has PST\@.  The adjacency spectrum is integral.  By
\cref{thm:parity-dichotomy}, PST under $A^k$ is therefore equivalent to PST
under $A$ when $k$ is odd and to PST under $A^2$ when $k$ is even.  These two
equivalences establish the classification in \cref{thm:hypercube}.

For the minimum positive transfer times, suppose first that $d$ is even.  Every
supported eigenvalue of $A^k$ is divisible by $2^k$, and both $0$ and $2^k$
occur.  Consequently
\[
 \gcd\{\lvert\eta-\eta_0\rvert:
          \eta\in\supp_{A^k}(e_x),\ \eta\ne\eta_0\}=2^k
\]
for any fixed $\eta_0\in\supp_{A^k}(e_x)$.  If $d$ and $k$ are odd, then all
supported eigenvalues of $A^k$ are odd, so every difference is even.  The
values $1$ and $-1$ both occur, and hence the gcd of the differences is $2$.
The formula for the minimum positive transfer time in
\cref{thm:integral-parity} now gives
$\pi/2^k$ and $\pi/2$, respectively.
\end{proof}

The theorem gives $\mathcal K(A(Q_d);x,\bar x)=\N^+$ when $d$ is even
and $2\N+1$ when $d$ is odd.

Throughout the remainder of this section, $A$ again denotes the ordinary
adjacency matrix.  The superscript $\T$ denotes transpose.  Laplacian and
signless Laplacian matrices occur only through the blocks $BB^{\T}$ and
$B^{\T}B$ in the proof for paths.

Godsil~\cite{Godsil2012} and Kendon and Tamon~\cite{KendonTamon2011} showed that
$C_4$ is the unique cycle with PST under the ordinary adjacency walk.  We
classify all powers of cycles and then classify PST under $A^2$ on paths.  The
exceptional behavior
of $P_7$ under higher powers is treated separately in \cref{sec:boundary}.

\subsection{Cycles}

\begin{lemma}\label{lem:cycle-antipodal-data}
If an adjacency power of $C_n$ has PST between distinct vertices, then
$n=2m$ and the vertices are antipodal.  For the pair $0,m$ the signed spectral
data are
\[
 \theta_j=2\cos\frac{\pi j}{m},\qquad
 \varepsilon_j=(-1)^j,\qquad 0\le j\le m.
\]
\end{lemma}
\begin{proof}
The power lies in the cycle scheme.  The valency-one obstruction
\cref{cor:scheme-valency-one} forces the antipodal relation, and Fourier
evaluation gives the displayed data.
\end{proof}

\begin{lemma}\label{lem:algebraic-integer-power}
Let $\alpha$ be a real algebraic integer whose conjugates lie in an interval
$I$ on which $x\mapsto x^r$ is injective.  If $\alpha^r\in\Q$, then
$\alpha\in\Z$.
\end{lemma}
\begin{proof}
Every conjugate $\alpha'$ satisfies $(\alpha')^r=\alpha^r$, and injectivity
gives $\alpha'=\alpha$.  Number fields are separable, so $\alpha$ has degree
one over $\Q$; being an algebraic integer, it belongs to $\Z$.
\end{proof}

\begin{lemma}\label{lem:cycle-small-parameters}
For $m\ge2$, put $\theta=2\cos(\pi/m)$.
If $r$ is odd and $\theta^r\in\Q$, then $m\in\{2,3\}$.  If
$(\theta^2)^r\in\Q$, then $m\in\{2,3,4,6\}$.
\end{lemma}
\begin{proof}
In the first case apply \cref{lem:algebraic-integer-power} on $\R$.
Then $\theta\in\Z$, and $0\le\theta<2$ gives $m=2,3$.
In the second case the conjugates of
$\theta^2=2+2\cos(2\pi/m)$ are nonnegative, so apply the lemma on
$[0,\infty)$.  Thus $2\cos(2\pi/m)$ is an integer, and
\[
\begin{array}{c|ccccc}
2\cos(2\pi/m)&-2&-1&0&1&2\\ \hline
m&2&3&4&6&1
\end{array}
\]
gives the result after excluding $m=1$.
\end{proof}

\begin{theorem}
\label{thm:cycle-powers}
Let $A$ be the adjacency matrix of $C_n$, $n\ge3$.  PST under $A^k$ between
distinct vertices occurs exactly for antipodal vertices of $C_4$, for every
$k\in\N^+$, and antipodal vertices of $C_8$, for even $k$.  The minimum times
are $\pi/2^k$ for $C_4$ and $\pi/2^q$ for $C_8$ when $k=2q$.
\end{theorem}
\begin{proof}
Use the data from \cref{lem:cycle-antipodal-data}.  If $k$ is odd, comparison
of the phases at $j=0,m$ gives $\tau/\pi\in\Q$, and comparison at $j=1,0$
then gives $\theta_1^k\in\Q$.  By
\cref{lem:cycle-small-parameters}, $m=2$ or $3$.  For $m=3$, the values
$2,1,-1,-2$ have signs $+,-,+,-$; relative to $2^k$, an equal-sign
difference and an opposite-sign difference are both odd, contradicting
\cref{thm:integral-parity}.  Hence only $m=2$, or $C_4$, remains.

Let $k=2q$.  If $m$ is odd, the equal values
$\theta_j^{2q}=\theta_{m-j}^{2q}$ have opposite signs, so PST is impossible.
For even $m$, the phases at $j=m/2,0$ give $\tau/\pi\in\Q$, and those at
$j=1,0$ give $\theta_1^{2q}\in\Q$.  The small-parameter lemma yields
$m\in\{2,4,6\}$.  For $m=6$, the signed values
$4^q,3^q,1,0$ have signs $+,-,+,-$.  The opposite-sign difference
$4^q-3^q$ and equal-sign difference $4^q-1$ are both odd, again contradicting
\cref{thm:integral-parity}.  Thus only $C_4$ and $C_8$ remain.

For $C_4$, the signed blocks are $2^k,0,(-2)^k$ with signs $+,-,+$, and
$\tau=\pi/2^k$ gives the required phases.  For $C_8$ and $k=2q$, the signed
values $4^q,2^q,0$ have signs $+,-,+$, and $\tau=\pi/2^q$ works.
Finally, the gcds of the supported differences are respectively $2^k$ and
$2^q$, so \cref{thm:integral-parity} gives the stated minimum times.
\end{proof}

\subsection{Paths}\label{sec:paths}

\begin{lemma}\label{lem:path-reflection}
If a positive power of $A(P_n)$ has PST between distinct vertices $a,b$,
then $a+b=n+1$.
\end{lemma}
\begin{proof}
By \cref{thm:strong-cospectral-exponents}, the vertices are strongly
cospectral for $A(P_n)$ and hence have equal diagonal entries in every
spectral idempotent. A positive eigenvector for the simple largest
adjacency eigenvalue has coordinates $x_s=\sin(s\pi/(n+1))$.
Its rank-one idempotent has diagonal entries proportional to $x_s^2$.
Since $s\pi/(n+1)\in(0,\pi)$, equality at two distinct coordinates
implies $a+b=n+1$.
\end{proof}

We first classify PST under $A^2$ for all paths.  The proof uses the following
block decomposition.

\begin{lemma}\label{lem:path-blocks}
Let $A_n=A(P_n)$.
\begin{enumerate}[label=\textup{(\roman*)}]
\item If $n=2m+1$, then after ordering the odd vertices before the even
vertices,
\[
 A_n^2=Q(P_{m+1})\oplus\bigl(2I_m+A(P_m)\bigr),
\]
where $D(P_{m+1})$ is the diagonal degree matrix and
$Q(P_{m+1})=D(P_{m+1})+A(P_{m+1})$ is the signless Laplacian.
\item If $n=2m$, then, after ordering the two color classes consecutively,
\[
 A_n^2=T_m\oplus RT_mR,
 \qquad
 T_m=2I_m+A(P_m)-e_1e_1^{\T},
\]
where $R$ is the $m\times m$ permutation matrix that reverses the coordinate
order.
\end{enumerate}
\end{lemma}

\begin{proof}
With respect to the bipartition of the path, write the adjacency matrix in the
following form:
\[
        A_n=\begin{pmatrix}0&B\\B^{\T}&0\end{pmatrix},
        \qquad
        A_n^2=\begin{pmatrix}BB^{\T}&0\\0&B^{\T}B\end{pmatrix}.
\]
Here $B$ is the biadjacency matrix from the odd color class to the even color
class.
If $n=2m+1$, then the odd part has $m+1$ vertices.  Two consecutive odd vertices
have one common even neighbor, and their diagonal entries count degrees in the
path.  Hence $BB^{\T}=Q(P_{m+1})$.  Each even vertex has degree two, and two
consecutive even vertices have one common odd neighbor, giving
$B^{\T}B=2I_m+A(P_m)$.

If $n=2m$, then both parts have size $m$.  On the part containing vertex $1$, the
diagonal entries are $1,2,\ldots,2$ and consecutive off-diagonal entries are
one.  Hence the block on the color class containing vertex $1$ is $T_m$; the
block on the other color class is its reversal.
\end{proof}

\begin{lemma}\label{lem:even-path-no-cospectral}
The matrix $T_m=2I_m+A(P_m)-e_1e_1^{\T}$ has no two distinct cospectral
coordinate vectors.
\end{lemma}

\begin{proof}
Set $C=T_m-2I_m=A(P_m)-e_1e_1^{\T}$.  Interpret $C$ as the adjacency matrix of
a path with a loop of weight $-1$ at vertex $1$.  Fix
$r\in\{1,\ldots,m\}$.  Every odd
closed walk must use the loop at vertex $1$.  The shortest such walk has length $2r-1$: it
travels directly from $r$ to $1$, uses the loop once, and returns directly.
The closed walk of length $2r-1$ is unique and has weight $-1$.  Therefore
\[
        (C^\ell)_{rr}=0\quad\text{for odd }\ell<2r-1,
        \qquad
        (C^{2r-1})_{rr}=-1.
\]
The least odd nonzero diagonal moment determines $r$, so distinct vertices
have different spectral measures for $C$.  Scalar shifts preserve
cospectrality: indeed,
\[
 (T_m^s)_{rr}=\sum_{j=0}^s\binom{s}{j}2^{s-j}(C^j)_{rr},
\]
and the inverse triangular relations recover the moments of $C$.  Hence the
vertices are not cospectral for $T_m=C+2I_m$.
\end{proof}

\begin{theorem}\label{thm:path-classification}
Let $n\geq2$, let $A$ be the adjacency matrix of $P_n$, and let $a,b$ be
distinct vertices.  PST under $A^2$ from $a$ to $b$ occurs if and only if
\[
  (n;a,b)=(3;1,3),\qquad (5;2,4),\qquad (7;2,6),
\]
up to interchanging $a$ and $b$.  The corresponding minimum positive transfer
times are respectively $\pi/2$, $\pi/2$, and $\pi/\sqrt2$.
\end{theorem}

\begin{proof}
The matrix $A^2$ preserves the two color classes, so PST can only occur
within one parity block.  Suppose first that $n=2m$.  By
\cref{lem:path-blocks}, each block is $T_m$ up to reversal.  PST requires
strong cospectrality, and hence cospectrality.  For distinct vertices,
cospectrality is excluded by \cref{lem:even-path-no-cospectral}.

Now let $n=2m+1$.  On the even vertices the block is
$2I_m+A(P_m)$.  Adding a scalar multiple of the identity changes only the
global phase, so $2I_m+A(P_m)$ has PST if and only if $A(P_m)$
does.  Coutinho, Juliano, and Spier~\cite{CoutinhoJulianoSpier2024} proved that
the only paths admitting PST for their adjacency matrices are $P_2$ and $P_3$;
in each case PST is between the end vertices.  The cases $P_2$ and $P_3$ give
$(P_5;2,4)$ and $(P_7;2,6)$,
respectively.

On the odd vertices the block is $Q(P_{m+1})$.  If $S$ is the bipartite sign
diagonal matrix, then
\[
 Q(P_{m+1})=SL(P_{m+1})S,\qquad
 \exp(-\ii tQ)=S\exp(-\ii tL)S.
\]
Thus diagonal switching preserves the existence and times of PST\@.  Coutinho and
Liu~\cite{CoutinhoLiu2015} proved that no tree other than $P_2$ has PST under its
Laplacian; hence $m+1=2$, which gives
$(P_3;1,3)$.

It remains to determine the minimum times.  For the end vertices of $P_2$, the
supported eigenvalues of $Q(P_2)$ are $2$ and $0$, and their projection signs
are opposite.  The phase congruences in \cref{thm:phase-congruences} therefore
give
\[
        2\tau\in\pi+2\pi\Z,
\]
so the minimum positive time is $\pi/2$.  For $A(P_2)$, the supported
eigenvalues $1$ and $-1$ also have opposite signs, and the same congruence again
gives the minimum time $\pi/2$.  For the end vertices of $P_3$, the supported
eigenvalues are $\sqrt2,0,-\sqrt2$, with signs $+,-,+$, respectively.  Hence
$\tau\sqrt2\in\pi+2\pi\Z$, and the minimum positive time is
$\pi/\sqrt2$.  The scalar shifts from $A(P_2)$ to $2I_2+A(P_2)$ and from
$A(P_3)$ to $2I_3+A(P_3)$ preserve all spectral differences and therefore
preserve these minimum times.
\end{proof}

The three exceptional paths in \cref{thm:path-classification} behave very
differently under higher powers.

\begin{corollary}\label{cor:P3-exponents}
Let $A=A(P_3)$.  Then
\[
        \mathcal K(A(P_3);1,3)=\N^+.
\]
For $r\in\N^+$, the minimum positive transfer time under $A^{2r}$ is $\pi/2^r$;
for $r\in\N$, the minimum positive transfer time under $A^{2r+1}$ is
$\pi/(2^r\sqrt2)$.
\end{corollary}

\begin{proof}
The common eigenvalue support for $A$ at vertices $1$ and $3$ is
$\{\sqrt2,0,-\sqrt2\}$.
Reflection has sign $+1$ at the two nonzero eigenvalues and sign $-1$ at
zero.  Hence $A^{2r}$ has supported eigenvalues $2^r$ and $0$, with signs
$+1$ and $-1$.  The phase relation is therefore
\[
        2^r\tau\in\pi+2\pi\Z,
\]
whose least positive solution is $\pi/2^r$.  Under $A^{2r+1}$, the supported
eigenvalues are $2^r\sqrt2$, $0$, and $-2^r\sqrt2$.  The two nonzero values
have sign $+1$, whereas $0$ has sign $-1$.  Thus
\[
        2^r\sqrt2\,\tau\in\pi+2\pi\Z,
\]
and the least positive solution is $\pi/(2^r\sqrt2)$.
\end{proof}

On $P_5$, PST from vertex $2$ to vertex $4$ under $A^2$ persists for every
positive even exponent.  The common spectral idempotents of the even powers do
not by themselves establish this persistence, because the phase equations
still depend on the exponent.  On the even color class, ordered as $(2,4)$,
one has
\begin{equation}\label{eq:P5-square-block}
        A(P_5)^2\big|_{\{2,4\}}=
        \begin{pmatrix}2&1\\1&2\end{pmatrix}.
\end{equation}

\begin{proposition}\label{prop:P5-even}
Let $A=A(P_5)$.  For every $r\in\N^+$, the path $P_5$ has PST from $2$ to $4$
under $A^{2r}$, and its minimum positive transfer time is
\[
        \tau_r=\frac{\pi}{3^r-1}.
\]
\end{proposition}

\begin{proof}
Since the even color class is invariant,
\[
 A(P_5)^{2r}\big|_{\{2,4\}}
 =\left(A(P_5)^2\big|_{\{2,4\}}\right)^r.
\]
The matrix in \eqref{eq:P5-square-block} has symmetric eigenvalue $3$ and
antisymmetric eigenvalue $1$.  Its $r$th power has corresponding eigenvalues
$3^r$ and $1$.  At the stated time their phases differ by
$\ee^{-\ii\tau_r(3^r-1)}=-1$, which gives PST between the coordinate states
indexed by vertices $2$ and $4$.  Conversely, PST requires the two phases to
differ by a factor of $-1$, and hence
\[
        \tau(3^r-1)\in\pi+2\pi\Z.
\]
The least positive solution is $\tau_r=\pi/(3^r-1)$.
\end{proof}

\begin{corollary}\label{cor:P5-exponents}
For vertices $2$ and $4$ of $P_5$,
\[
        \mathcal G(A(P_5);2,4)=\N^+,
        \qquad
        \mathcal K(A(P_5);2,4)=2\N^+.
\]
Thus the two vertices are strongly cospectral for every positive power, but
PST occurs only for even powers.
\end{corollary}

\begin{proof}
The common support and signs are
\[
\begin{array}{c|rrrr}
\lambda&\sqrt3&1&-1&-\sqrt3\\ \hline
\sigma_\lambda&+1&-1&-1&+1.
\end{array}
\]
Hence
\cref{thm:strong-cospectral-exponents} gives
$\mathcal G(A(P_5);2,4)=\N^+$.
All even exponents give PST by
\cref{prop:P5-even}.

Let $k$ be odd and suppose that $A(P_5)^k$ has PST at a positive time
$\tau$.  Equal phases on the two eigenvalues with sign $+1$ and on the two
eigenvalues with sign $-1$ would
give
\[
        \tau3^{k/2}\in\pi\Z,
        \qquad
        \tau\in\pi\Z.
\]
Writing $\tau=n\pi$ would force $n3^{k/2}\in\Z$, impossible for a nonzero
integer $n$ because $k$ is odd.  Hence no odd exponent has PST\@.
\end{proof}

The behavior of the remaining path $P_7$ is determined in
\cref{sec:boundary}.

\subsection{Johnson graphs}

Vinet and Zhan~\cite{VinetZhan2020} already characterize PST for every
real-weighted member of the Johnson scheme. The result below is a closed
classification within the particular sequence $A,A^2,A^3,\ldots$; it is not
a classification of previously untreated weighted Johnson Hamiltonians.

The restriction to relations of valency one also determines which powers of
the adjacency matrices of Johnson graphs admit PST\@.  The
vertices of $J(v,m)$ are the $m$-subsets of a fixed $v$-set.  Two vertices are
adjacent when their intersection has size $m-1$.  We use the convention
$1\leq m\leq v/2$.
By \cref{cor:scheme-valency-one}, only complementary pairs in $J(2m,m)$ need
be considered.  The proof first excludes all odd exponents for $m\geq2$.  For
even exponents, the parity of $m$ gives a short valuation obstruction, and the
valuation argument shows that $m=3$ is the only exceptional case.

\Needspace{14\baselineskip}
\begin{theorem}
\label{thm:johnson-powers}
Let $A$ be the adjacency matrix of $J(v,m)$ and let $r\in\N^+$.  PST under
$A^r$ between distinct vertices occurs if and only if one of the following
holds.
\begin{enumerate}[label=\textup{(\roman*)}]
\item $J(v,m)=J(2,1)=K_2$, the exponent $r$ is odd, and the minimum positive
      transfer time is $\pi/2$;
\item $J(v,m)=J(6,3)$, the exponent $r$ is even, and the two vertices are
      complementary $3$-subsets.  The minimum positive transfer time is
      \[
             \frac{\pi}{3^r-1}.
      \]
\end{enumerate}
\end{theorem}

The proof of \cref{thm:johnson-powers} is given in \cref{app:johnson-proof}.

\section{Isolated exponents from quadratic spectra}\label{sec:boundary}

A three-point spectrum identifies a general mechanism behind the exceptional
path $P_7$. Unlike the preceding results, its supported spectrum lies on a
rationally shifted quadratic lattice rather than a line $c\Z$.

\Needspace{16\baselineskip}
\begin{theorem}\label{thm:quadratic-triple}
Let $M$ be real symmetric and suppose distinct vertices $a,b$ are strongly
cospectral with support
\[
 \{a_0-b_0\sqrt\Delta,\ a_0,\ a_0+b_0\sqrt\Delta\},
 \qquad a_0,b_0\in\Q,\quad b_0>0,
\]
where $\Delta>1$ is a square-free integer. Suppose the signs at the two
outer values agree and are opposite to the sign at $a_0$. Then
\[
 \mathcal K(M;a,b)=
 \begin{cases}
 \N^+,&a_0=0,\\
 \{1\},&a_0\ne0.
 \end{cases}
\]
In the second case the minimum transfer time is $\pi/(b_0\sqrt\Delta)$.
No positivity assumption on $M$ or its supported eigenvalues is required.
\end{theorem}
\begin{proof}
Write $d=b_0\sqrt\Delta$. At exponent one, the two phase differences
from $a_0$ are $d$ and $-d$, both with the opposite sign. Thus the
minimum transfer time is $\pi/d$. If $a_0=0$, every positive exponent
works directly, or by \cref{thm:parity-dichotomy}.

Suppose $a_0\ne0$ and let $r\geq2$. Write
\[
 (a_0+d)^r=U_r+V_r\sqrt\Delta,\qquad
 (a_0-d)^r=U_r-V_r\sqrt\Delta,
 \quad U_r,V_r\in\Q.
\]
The odd terms in the binomial expansion show $V_r\ne0$: their
powers of $a_0$ all have parity $r-1$, so their nonzero terms have the
same sign. Likewise
\[
 U_r-a_0^r
 =\sum_{j=1}^{\lfloor r/2\rfloor}
   \binom r{2j}a_0^{r-2j}b_0^{2j}\Delta^j\ne0,
\]
since all summands have the same sign. If an outer powered value equals
$a_0^r$, incompatible signs already exclude PST. Otherwise PST would give
odd integers $u,v$ such that
$\tau((a_0+d)^r-a_0^r)=u\pi$ and
$\tau((a_0-d)^r-a_0^r)=v\pi$. Since $V_r\ne0$, we have $u\ne v$.
The following quotient would therefore equal $u/(u-v)\in\Q$:
\[
 \frac{(a_0+d)^r-a_0^r}{(a_0+d)^r-(a_0-d)^r}
 =\frac12+\frac{U_r-a_0^r}{2V_r\sqrt\Delta}.
\]
The right-hand side is irrational, a contradiction.
\end{proof}

For example, every nonzero rational shift $a_0I+tA(P_3)$ with positive
rational $t$ has endpoint exponent set $\{1\}$. A scalar shift preserves
PST for a fixed matrix, but need not preserve PST for its higher powers.

\begin{theorem}\label{thm:P7-exponents}
For $A=A(P_7)$,
\[
 \mathcal K(A(P_7);2,6)=\{2\}.
\]
The minimum transfer time under $A^2$ is $\pi/\sqrt2$.
\end{theorem}
\begin{proof}
On the even color class, ordered as $(2,4,6)$,
\cref{lem:path-blocks} gives $A^2=2I+A(P_3)$. The signed endpoint support
of this block is $2+\sqrt2,2,2-\sqrt2$ with signs $+,-,+$.
Theorem~\ref{thm:quadratic-triple} shows that among its powers precisely
the first has PST, at minimum time $\pi/\sqrt2$. Hence among the even
powers of $A$ precisely $A^2$ works. No odd power works by
\cref{cor:trees-odd}.
\end{proof}

The support at vertex $2$ of $P_7$ contains both
$\sqrt{2+\sqrt2}$ and $\sqrt{2-\sqrt2}$, whose ratio is
$1+\sqrt2$. Thus it is not contained in $c\Z$.
Together with \cref{cor:P3-exponents,cor:P5-exponents}, we obtain
\[
 \mathcal K(A(P_3);1,3)=\N^+,\quad
 \mathcal K(A(P_5);2,4)=2\N^+,\quad
 \mathcal K(A(P_7);2,6)=\{2\}.
\]
Failure of commensurability permits, but does not force, an exponent set
outside the four parity classes.

\section{Concluding remarks and problems}\label{sec:conclusion}

The parity theorem shows that, under the support hypothesis
$\supp_A(e_a)\cup\supp_A(e_b)\subseteq c\Z$, the existence of PST under $A^k$
depends only on whether $k$ is odd or even.  The graph classifications in
\cref{sec:applications} determine the resulting exponent sets for the
families considered here, while \cref{thm:path-classification} identifies all
paths having PST under $A^2$.

The path examples also mark the limit of the support hypothesis.  The common
support for $P_3$ is contained in $\sqrt2\Z$, whereas the corresponding
supports for $P_5$ and $P_7$ are not contained in $c\Z$ for any $c>0$.
Nevertheless, the exponent set for $P_5$ is determined only by parity, while
$P_7$ has the single PST exponent $2$.  Thus the support hypothesis is
sufficient for the parity theorem but does not characterize all pairs whose
exponent sets are determined only by parity.

The rational-matrix results narrow the remaining path problem. Odd exponents
are completely classified by \cref{cor:trees-odd}. Even paths admit no even
exponent: all even powers have the same strong-cospectrality obstruction as
$A^2$. On an odd path, every odd-indexed vertex supports the eigenvalue zero,
as the corresponding eigenvector has entries $\sin(s\pi/2)$. By
\cref{thm:zero-even,thm:path-classification}, an even exponent starting at
such a vertex can yield PST only for the endpoints of $P_3$.
Together with \cref{lem:path-reflection}, this leaves only reflected
even-indexed vertices on odd paths of order at least nine, for even
exponents greater than two.

Two problems remain directly connected to the classifications in this paper.

\begin{problem}\label{prob:exponent-sets}
Determine which subsets of $\N^+$ occur as
$\mathcal K(A(G);a,b)$ for a finite simple graph $G$ and distinct vertices
$a,b$.  In particular, determine which finite sets can occur.
\end{problem}

\begin{problem}\label{prob:path-powers}
Determine whether an even power $A(P_{2m+1})^{2r}$, with $m\geq4$ and
$r\geq2$, can have PST between reflected even-indexed vertices. The square is classified in
\cref{thm:path-classification}, while
\cref{cor:P3-exponents,cor:P5-exponents,thm:P7-exponents} exhibit three
different forms of persistence under higher powers.
\end{problem}

\appendix

\section{Association-scheme arguments}\label{app:association-schemes}

This appendix collects the association-scheme arguments used in
\cref{sec:applications} and gives the detailed proof for Johnson graphs.
The standard obstruction for adjacency matrices was developed by Coutinho,
Godsil, Guo, and Vanhove~\cite{CoutinhoEtAl2015}.  \Cref{prop:scheme-PQ} applies
to a real symmetric matrix
in a Bose--Mesner algebra.  The proposition also allows spectral idempotents to merge
under a function.  Since primitive idempotents have constant diagonal, the
second eigenmatrix gives an explicit condition on the spectral idempotents.
We follow the standard notation for association
schemes used by Bannai and Ito~\cite{BannaiIto1984}, Brouwer, Cohen, and
Neumaier~\cite{BrouwerCohenNeumaier1989}, and Godsil~\cite{Godsil1993}.
All matrices in this appendix are real symmetric matrices in the Bose--Mesner
algebra under discussion; their powers remain in the algebra.
The resulting restrictions do not apply to arbitrary matrices indexed by the
same vertex set.  The graph classifications in the main text use ordinary
adjacency matrices.

\subsection{Spectral idempotents in symmetric association schemes}

Let $\mathfrak X=(V,\{R_0,\ldots,R_d\})$ be a symmetric association scheme
with adjacency
matrices $A_0,\ldots,A_d$, primitive idempotents $E_0,\ldots,E_d$, first
eigenmatrix $P$, and second eigenmatrix $Q$, with conventions
\[
 A_i=\sum_{j=0}^d P_{ji}E_j,\qquad
 E_j=\frac1{|V|}\sum_{\ell=0}^d Q_{\ell j}A_\ell.
\]
Write $m_j=Q_{0j}=\operatorname{rank}(E_j)$.

Let
\[
        H=\sum_{i=0}^d c_iA_i\in\operatorname{span}_{\R}
        \{A_0,\ldots,A_d\},
        \qquad
        \theta_j=\sum_{i=0}^d c_iP_{ji},
\]
and let $f$ be real-valued on $\{\theta_0,\ldots,\theta_d\}$.  Fix
$\ell\in\{0,\ldots,d\}$.  For $\mu\in f(\{\theta_j\})$, put
\[
 J_\mu=\{j:f(\theta_j)=\mu\},\qquad
 M_\mu=\sum_{j\in J_\mu}m_j,\qquad
 R_{\ell\mu}=\sum_{j\in J_\mu}Q_{\ell j}.
\]

\begin{proposition}
\label{prop:scheme-PQ}
For $(a,b)\in R_\ell$, the matrix $f(H)$ has PST from $a$ to $b$ if and only if
\begin{enumerate}[label=\textup{(\roman*)}]
\item $|R_{\ell\mu}|=M_\mu$ for every $\mu$;
\item there are $\tau>0$ and $\gamma\in\C$, $|\gamma|=1$, such that
\[
        \ee^{-\ii\tau\mu}
        =\gamma\,\frac{R_{\ell\mu}}{M_\mu}
        \qquad\text{for every }\mu.
\]
\end{enumerate}
Moreover, condition~\textup{(i)} holds if and only if, for every $\mu$, there
is $\varepsilon_\mu\in\{1,-1\}$ such that
\begin{equation}\label{eq:scheme-termwise}
        Q_{\ell j}=\varepsilon_\mu m_j
        \qquad(j\in J_\mu).
\end{equation}
\end{proposition}

\begin{proof}
The spectral idempotent of $f(H)$ for $\mu$ is
$F_\mu=\sum_{j\in J_\mu}E_j$.  Since $(a,a)\in R_0$ and
$(a,b)\in R_\ell$,
\[
 \|F_\mu e_a\|^2=(F_\mu)_{aa}=\frac{M_\mu}{|V|},
 \qquad
 \langle F_\mu e_a,F_\mu e_b\rangle
 =(F_\mu)_{ab}=\frac{R_{\ell\mu}}{|V|}.
\]
$F_\mu e_a$ and $F_\mu e_b$ have equal norm.  Equality in Cauchy--Schwarz shows
that they are parallel if and only if
$|R_{\ell\mu}|=M_\mu$; when $|R_{\ell\mu}|=M_\mu$,
\[
        F_\mu e_b=\frac{R_{\ell\mu}}{M_\mu}F_\mu e_a,
        \qquad
        \frac{R_{\ell\mu}}{M_\mu}\in\{1,-1\}.
\]
For $j\in J_\mu$, applying $E_j$ to
$F_\mu e_b=(R_{\ell\mu}/M_\mu)F_\mu e_a$ gives
\[
 E_je_b=\frac{R_{\ell\mu}}{M_\mu}E_je_a.
\]
Thus $a,b$ are strongly cospectral for $H$, with signs constant on each
fiber $J_\mu$.  \Cref{prop:spectral-function-pst} now gives
conditions~\textup{(i)} and~\textup{(ii)}.

For the final equivalence, the vectors $E_j e_a$ and $E_j e_b$ have squared
norm $m_j/|V|$ and inner product $Q_{\ell j}/|V|$.  Cauchy--Schwarz therefore
gives
\[
 |Q_{\ell j}|=|V|\,|(E_j)_{ab}|
 \leq |V|\sqrt{(E_j)_{aa}(E_j)_{bb}}=m_j.
\]
Consequently,
\[
 |R_{\ell\mu}|
 \leq\sum_{j\in J_\mu}|Q_{\ell j}|
 \leq\sum_{j\in J_\mu}m_j=M_\mu.
\]
The equality $|R_{\ell\mu}|=M_\mu$ holds if and only if every
intermediate inequality is an equality and all $Q_{\ell j}$, $j\in J_\mu$,
have the same sign.  Equality throughout the two inequalities is therefore
equivalent to \eqref{eq:scheme-termwise}.
\end{proof}

For $f(x)=x^k$, the values $\theta_j^k$ determine the fibers $J_\mu$, whereas
sums of the corresponding entries of $Q$ determine strong cospectrality for
$H^k$.  Equality in the resulting triangle inequality requires all summands
to have the same extremal sign.

Write $k_\ell=P_{0\ell}$ for the valency of $R_\ell$.

\begin{corollary}
\label{cor:scheme-valency-one}
Suppose that $a\ne b$ and $(a,b)\in R_\ell$.  The following are equivalent:
\begin{enumerate}[label=\textup{(\roman*)}]
\item $E_j e_a$ and $E_j e_b$ are parallel for every $j$;
\item
\[
        Q_{\ell j}\in\{m_j,-m_j\}
\]
for every $j$;
\item $k_\ell=1$.
\end{enumerate}
Consequently, if $H^k$ has PST from $a$ to $b$ for some $k\in\N^+$, then
$R_\ell$ has valency one.
\end{corollary}

\begin{proof}
For each $j$, the vectors $E_j e_a$ and $E_j e_b$ have squared norm
$m_j/|V|$ and inner product $Q_{\ell j}/|V|$.  They are parallel if and only if
$|Q_{\ell j}|=m_j$, which proves the equivalence of \textup{(i)} and
\textup{(ii)}.  The standard
eigenmatrix identity
\[
        m_jP_{j\ell}=k_\ell Q_{\ell j}
\]
shows that \textup{(ii)} forces every eigenvalue of $A_\ell$ to have absolute
value $k_\ell$.  Hence $A_\ell^2=k_\ell^2I$.  The diagonal entries of $A_\ell^2$ are also
equal to $k_\ell$, and therefore $k_\ell=k_\ell^2$.  Since $R_\ell$ is
nonempty, $k_\ell=1$.

Conversely, $a\ne b$ implies $\ell\ne0$, so $A_\ell$ has zero diagonal.
If $k_\ell=1$, it is the permutation matrix of a fixed-point-free involution.  Thus $A_\ell^2=I$ and
$P_{j\ell}\in\{1,-1\}$ for every $j$.  The eigenmatrix identity gives
$Q_{\ell j}=m_jP_{j\ell}\in\{m_j,-m_j\}$.
If $H^k$ has PST from $a$ to $b$, then \eqref{eq:scheme-termwise}, applied
with $f(x)=x^k$, gives condition~\textup{(ii)}, and hence $k_\ell=1$.
\end{proof}

Thus PST under a positive power can occur only between vertices in a relation
of valency one.  On such a relation, the signs must be constant whenever
$\theta_i^k=\theta_j^k$.

\subsection{Thin relations}

A relation of valency one is called a \emph{thin relation}.

Assume $k_\ell=1$ and set $\varepsilon_j=P_{j\ell}\in\{1,-1\}$.

Fix $k\in\N^+$, and for each value $\mu$ among the $\theta_j^k$ put
$J_\mu^{(k)}=\{j:\theta_j^k=\mu\}$.

\begin{corollary}\label{cor:thin-relation-pst}
The matrix $H^k$ has PST between the two vertices in every
edge of $R_\ell$ if and only if
\begin{enumerate}[label=\textup{(\roman*)}]
\item the signs $\varepsilon_j$ are constant on every set $J_\mu^{(k)}$,
      say with common value $\varepsilon_\mu$;
\item there are $\tau>0$ and $|\gamma|=1$ such that
\[
        \ee^{-\ii\tau\mu}=\gamma\varepsilon_\mu
        \qquad(\mu\in\{\theta_0^k,\ldots,\theta_d^k\}).
\]
\end{enumerate}
Thus the existence and time of PST depend only on the thin relation and
the transformed eigenvalues, not on the chosen edge of $R_\ell$.
\end{corollary}

\begin{proof}
When $k_\ell=1$, the eigenmatrix identity gives
$Q_{\ell j}/m_j=P_{j\ell}=\varepsilon_j$.  Substituting
$Q_{\ell j}/m_j=\varepsilon_j$ into \cref{prop:scheme-PQ} with $f(x)=x^k$
gives the result.
\end{proof}

\begin{corollary}
\label{cor:distance-regular-antipodal}
Let $X$ be a connected distance-regular graph with diameter $D$ and distance
relations $R_0,\ldots,R_D$, and let $K$ be a real symmetric matrix in its
Bose--Mesner algebra.  If $K^k$ has PST between distinct vertices $a$ and $b$
for some $k\in\N^+$, then $(a,b)$ lies in a distance relation of valency one.
Consequently, if every nontrivial distance relation has valency greater than
one, then no positive power of $K$ has PST between distinct vertices.  In
particular, if the only nontrivial relation of valency one is $R_D$, then every
transfer pair under a positive power is antipodal.
\end{corollary}

\begin{proof}
The distance relations of a distance-regular graph form a symmetric
association scheme.  If $(a,b)\in R_\ell$ and $K^k$ has PST from $a$ to
$b$, then \cref{cor:scheme-valency-one} forces $k_\ell=1$.  If every
nontrivial distance relation has valency greater than one, no distinct pair can
satisfy this necessary condition.  If $R_D$ is the only nontrivial relation of
valency one, then every transfer pair must lie in $R_D$ and is therefore
antipodal.
\end{proof}

For powers, the first eigenmatrix determines which primitive idempotents are
merged by $x\mapsto x^k$, and the second eigenmatrix determines whether the
corresponding vertex projections are parallel.  Coutinho, Godsil, Guo, and
Vanhove~\cite{CoutinhoEtAl2015} used these two roles of the eigenmatrices in the
Bose--Mesner idempotent method for state transfer.

\subsection{Hamming schemes}\label{app:hamming-schemes}

We also record the Bose--Mesner-algebra form of the obstruction used in the
proof of \cref{prop:hamming-adjacency-powers}.

\begin{theorem}
\label{thm:nonbinary-hamming}
Let $q\geq2$, let $\mathcal A$ be the Bose--Mesner algebra of $H(d,q)$, and let
$K$ be a real symmetric matrix in $\mathcal A$.  If $q\geq3$, then no positive
power of $K$ has PST between distinct vertices.

For $q=2$, two distinct vertices can be strongly cospectral for $K^k$, for
some $k\in\N^+$, only if they are antipodal.  For an antipodal pair, the
signs on the primitive idempotents are $(-1)^j$.
\end{theorem}

\begin{proof}
Fix distinct vertices $a,b$ at Hamming distance $\ell>0$, and fix
$k\in\N^+$.  If $q\geq3$, then
$k_\ell=\binom d\ell(q-1)^\ell>1$.

Suppose that $a$ and $b$ are strongly cospectral for $K^k$.  Write
\[
        K=\sum_{j=0}^d\kappa_jE_j.
\]
For an eigenvalue $\eta$ of $K^k$, its spectral idempotent is
\[
        F_\eta=\sum_{j:\,\kappa_j^k=\eta}E_j.
\]
If $F_\eta e_b=\varepsilon_\eta F_\eta e_a$ and
$\kappa_j^k=\eta$, then $E_jF_\eta=E_j$, and hence
\[
        E_j e_b=\varepsilon_\eta E_j e_a.
\]
Thus $E_j e_a$ and $E_j e_b$ are parallel for every $j$.
By \cref{cor:scheme-valency-one}, the distance-$\ell$ relation must have
valency one.  For $q\geq3$, the inequality $k_\ell>1$ contradicts the
valency-one requirement, so no positive power of $K$ has PST between distinct vertices.
When $q=2$, the only distance relations
of valency one are $R_0$ and $R_d$, so distinct strongly cospectral vertices
for a power of $K$ must be antipodal.

If $a$ and $b$ are antipodal, then $b-a$ is the all-ones vector.  Every binary
character of weight $j$ takes the value $(-1)^j$ on the all-ones vector, and hence
$E_j e_b=(-1)^jE_j e_a$.
\end{proof}

\subsection{Johnson graphs}\label{app:johnson-proof}

\begin{proof}[Proof of \cref{thm:johnson-powers}]
The relation of Johnson distance $\ell$ has valency
\[
        k_\ell=\binom m\ell\binom{v-m}{\ell}.
\]
For $\ell>0$, the valency $k_\ell$ equals one only when $v=2m$ and $\ell=m$.  By
\cref{cor:scheme-valency-one}, a transfer pair must therefore consist of
complementary $m$-subsets in $J(2m,m)$.  The eigenvalues of its adjacency
matrix and the complement signs on the primitive idempotents are the standard
Johnson-scheme formulas~\cite{BrouwerCohenNeumaier1989,Godsil1993}:
\[
        \theta_j=(m-j)^2-j,\qquad
        \varepsilon_j=(-1)^j,\qquad 0\leq j\leq m.
\]

The adjacency spectrum is integral, so \cref{thm:parity-dichotomy} reduces
all odd exponents to $A$ and all even exponents to $A^2$.

When $m=1$, the graph is $K_2$: precisely the odd powers have PST, with
minimum positive transfer time $\pi/2$.  Let $m\geq2$.  For $A$, the two differences
from $\theta_0=m^2$ are
\[
        \theta_1-\theta_0=-2m,\qquad
        \theta_2-\theta_0=-2(2m-1),
\]
whose gcd is two.  Moreover,
\[
        \theta_j-\theta_0=-j(2m-j+1)
        \qquad(0\leq j\leq m),
\]
and $j(2m-j+1)$ is even.  Thus every difference
$\theta_j-\theta_0$ is divisible by two, while the first two nonzero
differences have gcd two.  The gcd of all supported differences is therefore
two.  The quotient
$(\theta_1-\theta_0)/2=-m$ would have to be odd because
$\varepsilon_1=-1$, while
$(\theta_2-\theta_0)/2=-(2m-1)$ would have to be even because
$\varepsilon_2=1$.  Since $-(2m-1)$ is odd,
\cref{thm:integral-parity} excludes PST for $A$, and hence for every odd
power.

To settle the even powers, PST for $A^2$ would require
\[
 \ee^{-\ii\tau\theta_j^2}=\gamma(-1)^j\qquad(0\le j\le m).
\]
If merged values have incompatible signs, PST is already impossible;
otherwise these identities remain necessary.  Put
$D_j=\theta_j^2-\theta_0^2$.  The integral parity condition requires the differences $D_j$ with $j$ odd to have least
$2$-adic valuation, and requires every $D_j$ with even $j>0$ to have strictly
larger valuation.

If $m$ is even, then
\[
 \theta_0=m^2\equiv0\pmod4,\qquad
 \theta_1=m(m-2)\equiv0\pmod4,\qquad
 \theta_2=(m-2)^2-2\equiv2\pmod4.
\]
Thus $D_1\equiv0\pmod{16}$, whereas $D_2\equiv4\pmod{16}$, and hence
\[
        \nu_2(D_1)\geq4>2=\nu_2(D_2),
\]
the reverse of the required inequality.
The inequality $\nu_2(D_1)>\nu_2(D_2)$ excludes $m=2$, $m=4$, and every
larger even value of $m$.

Assume next that $m\geq5$ is odd and put $t=\nu_2(m-1)$.  All relevant
eigenvalues are odd.  Factoring the two differences gives
\[
 \theta_0-\theta_1=2m,\quad
 \theta_0+\theta_1=2m(m-1),\quad
 \theta_0-\theta_4=4(2m-3),\quad
 \theta_0+\theta_4=2(m^2-4m+6).
\]
Consequently,
\[
        \nu_2(D_1)=2+t\geq3,
        \qquad
        \nu_2(D_4)=3.
\]
The difference $D_1$, which corresponds to sign $-1$, therefore does not have
strictly smaller valuation than the difference $D_4$, which corresponds to
sign $+1$.  Hence $A^2$ has no PST\@.

The only remaining case is $m=3$.  Its eigenvalues and signs are
\[
\begin{array}{c@{\qquad}rrrr}
j&0&1&2&3\\
 \theta_j&9&3&-1&-3\\
 \varepsilon_j&1&-1&1&-1.
 \end{array}
\]
For the square, the values on the four supported idempotents are $81,9,1,9$,
and time $\pi/8$ gives equal phases
at $81$ and $1$ and the opposite phase at $9$.  Thus $A^2$ has PST, and
\cref{thm:parity-dichotomy} shows that every positive even power has PST\@.
For an even exponent $r$, put $x=3^r$.  The distinct supported eigenvalues of
$A^r$ are
$x^2,x,1$, and
\[
 \gcd(x^2-x,x^2-1)=x-1=3^r-1.
\]
Hence \cref{thm:integral-parity} gives the minimum positive transfer time
$\pi/(3^r-1)$.
\end{proof}

\begingroup
\sloppy

\endgroup

\end{document}